\documentclass[12pt]{article}

\usepackage{moreverb}
\usepackage{amsmath,amsthm}
\usepackage{appendix}
\usepackage{amssymb,latexsym}
\usepackage{amsfonts}
\usepackage{amscd}
\usepackage{times}
\usepackage{graphics}
\usepackage{color,soul}
\usepackage{tikz}
\usetikzlibrary{decorations.pathmorphing} 
\usetikzlibrary{matrix} 
\usetikzlibrary{arrows} 
\usetikzlibrary{calc} 

\tikzstyle{block} = [draw,rectangle,thick,minimum height=2em,minimum width=2em]
\tikzstyle{sum} = [draw,circle,inner sep=0mm,minimum size=2mm]
\tikzstyle{connector} = [->,thick]
\tikzstyle{line} = [thick]
\tikzstyle{branch} = [circle,inner sep=0pt,minimum size=1mm,fill=black,draw=black]
\tikzstyle{guide} = []
\tikzstyle{snakeline} = [connector, decorate, decoration={pre length=0.2cm,
	post length=0.2cm, snake, amplitude=.4mm,
	segment length=2mm},thick, magenta, ->]
\tikzstyle{output} = [coordinate]
\tikzstyle{input} = [coordinate]

\newtheorem{lem}{Lemma}

\newtheorem{thm}{Theorem}
\newtheorem{rmk}{Remark}[section]

\newcommand\BibTeX{{\rmfamily B\kern-.05em \textsc{i\kern-.025em b}\kern-.08em
T\kern-.1667em\lower.7ex\hbox{E}\kern-.125emX}}

\begin{document}


\title{Input-Delay Compensation in a Robust Adaptive Control Framework}

\author{Kim-Doang Nguyen, Harry Dankowicz\\Department of Mechanical Science and Engineering, \\University of Illinois at Urbana-Champaign, USA\\Email: knguyen9@illinois.edu,~danko@illinois.edu}
\maketitle

\begin{abstract}
	A modification to the ${\cal L}_1$ control framework for uncertain systems with actuator delay is presented. Specifically, a time delay is introduced in the control input of the state predictor to compensate for the destabilizing effect of input delay in the plant. For this modified framework, the analysis shows that the output of the adaptive system closely follows the behavior of a suitably defined, nonadaptive, stable reference system provided that a delay-dependent stability condition is satisfied and the adaptive gain is chosen sufficiently large. The set of combinations of input delay and compensation delay for which the stability condition is satisfied contains an open set of pairs of positive values provided that a filter bandwidth, characteristic of $L_1$ adaptive control is chosen sufficiently large. The efficacy of the delay compensation is illustrated by a simple example. A numerical continuation is also performed to explore the stability region for a case where this can be approximated a priori.

\end{abstract}

\vspace{-6pt}

\section{Introduction}
\vspace{-2pt}
Time delays are unavoidable in many practical applications. For example, in networked systems, the time each node takes to communicate with its neighbors is dependent on the physical distance and the technical specifications of the system. These communication delays affect the collective performance  of the network. Delay can be a result of measurements, which are often used as feedback. Feedback delay may harm not only the performance but also the stability of a control system. 

Another common delay type is input delays. These are caused, for example, by actual transport in chemical processes or by the internal dynamics of actuators. Input delays are known to have strong destabilizing effects on control systems and pose challenging problems. Early work sought compensation for input delay is based on Smith predictor \cite{Smith1957} for single-input-single-output open-loop stable systems. Finite Spectrum Assignment \cite{Kwon1980} and Artstein reduction \cite{Artstein1982} extended Smith predictor to multi-input-multi-output open-loop unstable systems. However, the efficacy of these methods requires high fidelity of the system model and the exact knowledge of the delay value \cite{Krstic2009}. 

More recent techniques include robust and adaptive control schemes that are able to deal with uncertainties and unknown delays in the systems. For example, the research documented in \cite{Yue2004,Parlakci2006,Cacace2016} proposes robust controllers to deal with input delay for uncertain systems. Delay-dependent state feedback $H_\infty$ controllers is designed in \cite{Zhang2003,Du2007,Parlakci2011} for linear systems with actuator delay. Model transformation is a common technique used in the adaptive control of systems with input delays \cite{Evesque2003,Niculescu2003}. However, this requires exact information of the delay value in the control laws. In recent important developments on the topic, \cite{Krstic2009,Bresch-Pietri2014,Basturk2015} develop a control framework in terms of predictor-based formulation to transform systems with input delays into partial differential equations of transport type. Work in \cite{Chen2008} concerns with the problem of guaranteed cost control for a class of fuzzy systems with input delay. In \cite{Rehan2016}, a state feedback scheme that depends on the range of delay values is constructed for system with constraints in the control input. A finite-frequency approach is formulated in \cite{Sun2012} for active suspension systems with actuator delays. Work in \cite{Polyakov2013} proposes an interval observation technique for the design of a linear feedback scheme to stabilize linear-time invariant system with a time-varying input delay. See \cite{Krstic2009,Richard2003,Niculescu2004} for extensive reviews on control techniques for delay systems.

We are interested in designing an input-delay compensation scheme that can adapt quickly to the system uncertainty. This problem is challenging because it is well known that fast adaptation often degrades the system's ability to tolerate time delays in the control loop, especially input delays. Work in \cite{Hovakimyan2010book} develops a control framework that facilitates fast adaptation with guaranteed robustness. The applications of the control framework in flight control systems is described in \cite{Hovakimyan2011}. The delay robustness of the framework is analyzed numerically in \cite{Nguyen2013} and rigorously in \cite{Nguyen2015tdm}. 

In this paper, we present the analysis for a modification to this control architecture to accommodate large actuator delays by purposely introducing on a delay in the state predictor. This modification was suggested to the authors by Dr. Naira Hovakimyan. We establish a stability condition that is a function of the actuator delay as well as the delay-compensation term. The formulation suggests that there exists a range for the compensation delay around the input delay on which the condition is satisfied. Within this range, the closed-loop adaptive system is rigorously proven to admit a transient performance bound via the bounded-input-bounded-output stability of a non-adaptive reference system. 

The remainder of this paper is organized as follows. The problem statement is presented in Section~\ref{8sec:plant}. In Section~\ref{8sec:refsys}, we design a nonadaptive reference system, formulate a delay-dependent stability condition, and show the bounded-input-bounded-output stability of the reference system for certain values of the delays. Sections~\ref{8sec:transbnd} and \ref{sec:trans} discuss the design of the adaptive controller and establish the transient performance bounds when the stability condition is satisfied. The effect of the input delay on the system response, as well as the efficacy of the delay compensation is illustrated in Section~\ref{8sec:num}. Here, we also present the numerical analysis based on Pad\'{e} approximants and continuation techniques to construct an approximation for the region of values of the input and compensation delays where the stability condition is satisfied. Section~\ref{8conc} gives concluding remarks.

\vspace{-6pt}

\section{The Open-Loop Plant}
\label{8sec:plant}
Consider the following system
\begin{align}
	\label{inpdelay}
	\dot x(t)&=A_mx(t)+b\big(u(t-\tau)+\theta^\top(t) x(t)+\sigma(t)\big),\nonumber\\ 
	x(0)&=x_0,\nonumber\\
	y(t)&=c^\top x(t)
\end{align}
where $x\in\mathbb{R}^n$, $u\in \mathbb{R}$, $y\in \mathbb{R}$, $A_m\in\mathbb{R}^{n\times n}$ is a Hurwitz matrix, $b\in \mathbb{R}^{n\times1}$, $c\in \mathbb{R}^{1\times n}$, $\tau$ is an unknown input delay, $\theta(t)\in \mathbb{R}^n$ and $\sigma(t)\in \mathbb{R}$ are unknown time-varying parameters such that $\|\theta(t)\|_2\le\theta_b$ and $|\sigma(t)|\le\sigma_b$. 

The control objective is for the system output $y$ to track a desired trajectory $y_\mathrm{des}$, which is the output of the following desired system with $y_d$ being a reference trajectory:
\begin{align}
	\label{desPDE}
	\dot x_\mathrm{des}(t)&=A_mx_\mathrm{des}(t)+b y_d(t), ~x_\mathrm{des}(0)=x_{\mathrm{des},0},\nonumber\\
	y_\mathrm{des}(t)&=c^\top x_\mathrm{des}(t).
\end{align}
Similar control objectives were shown in \cite{Hovakimyan2010book} to be achievable using the ${\cal L}_1$ adaptive control framework with fast adaptation for $\tau$ less than some critical value. In this paper, we present and investigate a modification to the ${\cal L}_1$ framework that appears to support stable operation with even larger value of the input delay.

\section{Nonadaptive Reference System}
\label{8sec:refsys}
In this section, we analyze a nonadaptive reference system that represents the ideal behavior for the desired closed-loop control system. The control law of this reference system includes a delay-compensation term. The analysis establishes a delay-dependent stability condition for the reference system. It then shows that, for sufficiently large values of a design parameter, there exists a range for the input delay $\tau$ and a range for the compensation delay $\hat{\tau}$ around the value of $\tau$, where the stability condition is satisfied.

Consider a reference system with state $x_\mathrm{ref}(t)$ and output $y_\mathrm{ref}(t)$, identical to \eqref{inpdelay}, but with input $u_{\mathrm{ref}}(t)$ given by the solution to the delay-differential equation
\begin{align}
	\label{uref}
	\dot u_{\mathrm{ref}}(t)&=-k\big(u_\mathrm{ref}(t-\tau)-u_\mathrm{ref}(t-\hat\tau)+u_\mathrm{ref}(t)+\theta^\top(t) x_\mathrm{ref}(t)+\sigma(t)-k_dy_d(t)\big)\\
	u_\mathrm{ref}(t)&=0~\forall t\in[-\max\{\tau,\hat{\tau}\},0],\nonumber
\end{align}
Here, the delay $\hat{\tau}$ is introduced to alleviate the destabilizing effect of the input delay $\tau$. It is trivial to see that when $\hat{\tau}=0$, the reference system recovers a special case of the one analyzed in \cite{Nguyen2015tdm}. Moreover, when $\hat{\tau}=\tau$, equation \eqref{uref} becomes a stable ordinary differential equation.

\subsection{Delay-dependent stability condition}
\label{sec: stab}
In the frequency domain,
\begin{align}
	\label{8usdef}
	u_{\mathrm{ref}}(s)= F(s;\tau,\hat\tau)\big(\eta_{\mathrm{ref}}(s)-k_dy_d(s)\big),
\end{align}
where
\begin{align}
	\label{8urefLap}
	F(s;\tau,\hat\tau):=-k\big(s+k \mbox{e}^{-\tau s}-k \mbox{e}^{-\hat\tau s}+k\big)^{-1}
\end{align}
and $\eta_{\mathrm{ref}}(s)$ and $y_d(s)$ are the Laplace transforms of  $\eta_{\mathrm{ref}}(t):=\theta^\top(t) x_\mathrm{ref}(t)+\sigma(t)$ and $y_d(t)$, respectively. We can see from \eqref{8urefLap} that since $F(s;\hat\tau,\hat\tau)$ has only one pole $s=-k$, which is negative, $F(s;\hat\tau,\hat\tau)$ is exponentially stable. Furthermore, the rightmost roots of an LTI delay differential equation depend continuously on the delay \cite{Michiels2007}, pp.~11-12. Therefore, for a given value of $\tau$, there exist $\underline \tau$ and $\bar \tau$ such that $F(s;\tau,\hat\tau)$ is an exponentially stable transfer function for all $\hat\tau\in[\tau-\underline \tau,\tau+\bar \tau]$.

We have
\begin{align}
	\label{8gdef}
	g(\tau,\hat\tau):=\| F(s;\tau,\hat\tau)\|_{{\cal L}_1}
\end{align}
is finite for all $\hat\tau\in[\tau-\underline \tau,\tau+\bar \tau]$ because on this interval $F(s;\tau,\hat\tau)$ is exponentially stable. Hence, on this interval,
\begin{equation}
	\|u_{\mathrm{ref}}\|_{{\cal L}_\infty}\leq g(\tau,\hat\tau)\big(\|\eta_{\mathrm{ref}}\|_{{\cal L}_\infty}+k_d\|y_d\|_{\mathcal{L}_\infty}\big).
	\label{8eq:ubound}
\end{equation}

Similarly, it follows from \eqref{8usdef} that
\begin{align}
	\label{8riomap}
	x_{\mathrm{ref}}(s)= \Phi(s;\tau,\hat\tau)\eta_{\mathrm{ref}}(s)&+\Psi(s;\tau,\hat\tau)y_d(s)+(s\mathbb{I}-A_m)^{-1}x_0,
\end{align}
where $ H(s):=(s\mathbb{I}-A_m)^{-1}b$, and
\begin{align}
	\label{8Phi}
	\Phi(s;\tau,\hat\tau)&:= H(s)\big(1+\mbox{e}^{-\tau s}  F(s;\tau,\hat\tau)\big),\\
	\Psi(s;\tau,\hat\tau)&:= -H(s)\mbox{e}^{-\tau s}  F(s;\tau,\hat\tau)k_d.
\end{align}
Using the same method that proves Lemmas~3 and 4 in \cite{Nguyen2015tdm}, we have
\begin{align}
	\label{8fd}
	f(\tau,\hat\tau):=\| \Phi(s;\tau,\hat\tau)\|_{{\cal L}_1}
\end{align}
is continuous in $\hat\tau$ on $[\tau-\underline \tau,\tau+\bar \tau ]$. Furthermore, 
\begin{align}
	\label{fbar}
	\bar{f}(\tau):=f(\tau,\tau)
\end{align}
is continuous in $\tau$ for all $\tau$.

We proceed to show that $\bar f(0)<\infty$ and is inversely proportional to the filter bandwidth $k$. To this end, consider the special case that $\hat{\tau}=\tau=0$, $x_0=0$ and $y_d(t)=0$. It follows from \eqref{8riomap} that
\begin{align}
	\label{8refsys1}
	\|x_{\mathrm{ref}}\|_{{\cal L}_\infty}&\leq \left\|s(s\mathbb{I}-A_m)^{-1}\right\|_{{\cal L}_1}\|b\|\left\|\frac{1}{s+k}\right\|_{{\cal L}_1}\left\|\eta_{\mathrm{ref}}\right\|_{{\cal L}_\infty}\nonumber\\
	&\leq \Big(1+\left\|A_m(s\mathbb{I}-A_m)^{-1}\right\|_{{\cal L}_1}\Big)\frac{\|b\|}{k}\left\|\eta_{\mathrm{ref}}\right\|_{{\cal L}_\infty}.
\end{align}
Therefore,
\begin{align}
	\label{8refsys2}
	\bar f(0)\leq\Big(1+\left\|A_m(s\mathbb{I}-A_m)^{-1}\right\|_{{\cal L}_1}\Big)\frac{\|b\|}{k}.
\end{align}
where $\left\|A_m(s\mathbb{I}-A_m)^{-1}\right\|_{{\cal L}_1}$ is finite because $A_m$ is a Hurwitz matrix.

The formulation of the delay-dependent stability condition below is inspired by the continuity argument discussed in \cite{Naghnaeian2012}.
\begin{lem}
	\label{lem:stbcnd}
	\textit{There exist values of $k$, $\tau_s$, $\underline \delta\le\underline{\tau}$ and $\bar\delta\le\bar{\tau}$ such that
		\begin{align}
			\label{8dm}
			f(\tau,\hat \tau)\theta_b&< 1,~~\forall \tau\in[0,\tau_s]\mbox{ and }\hat{\tau}\in[\tau-\underline{\delta}, \tau+\bar {\delta}].
		\end{align}}
	\end{lem}
	\begin{proof}
		Since $\bar f(0)\to 0$ as $k\to\infty$ per \eqref{8refsys2}, it follows that there exists a $K$, such that $k>K$ implies that
		\begin{eqnarray}
			\label{8L1cond}
			\bar f(0)\theta_b<1.
		\end{eqnarray}
		For such a $k$, because of the continuity of $\bar f(\tau)$ in $\tau$, there exists a $\tau_s$ such that 
		\begin{eqnarray}
			\label{8L1condfbar}
			f(\tau,\tau)\theta_b=\bar f(\tau)\theta_b<1, ~\forall \tau\in[0,\tau_s]
		\end{eqnarray}
		Similarly, the claim now follows by the continuity of $f(\tau,\hat \tau)$ in $\hat{\tau}$.
	\end{proof}
	
	In Section~\ref{8ssec:stb} below, we show that the closed-loop reference system admits a transient performance bound for all $\tau\in[0,\tau_s]$ and $\hat{\tau}\in[\tau-\underline{\delta}, \tau+\bar {\delta}]$. Furthermore, in Section~\ref{sec:trans}, we show that, for sufficiently large adaptive gains, the state and control input of the closed-loop adaptive control system follow those of the reference system closely. 
	\subsection{Transient performance of the reference system}
	\label{8ssec:stb}
	In this section, we show that the closed-loop reference system may be designed to admit a transient performance bound even for nonzero input delay.
	
	Consider the delay-dependent norms
	\begin{align}
		\label{8rhod}
		\rho_d(\tau,\hat\tau):= \| \Psi(s;\tau,\hat\tau)\|_{{\cal L}_1}\|y_d\|_{{\cal L}_\infty}
	\end{align}
	and
	\begin{equation}
		\rho_\mathrm{ic}:= \|(s\mathbb{I}-A_m)^{-1}x_0\|_{{\cal L}_1},
	\end{equation}
	which represent the effects of the desired trajectory $y_d$ and the initial condition $x_0$ on the solution to \eqref{8riomap}. Here, $\rho_d(\tau,\hat\tau)$ and $\rho_\mathrm{ic}$ are both finite for all $\hat\tau\in[\tau-\underline \delta,\tau+\bar \delta ]$, since $F(s;\tau,\hat\tau)$ is exponentially stable on this delay interval and $A_m$ is Hurwitz. 
	
	\begin{thm}
		\label{8thm1}
		\textit{Suppose $k$ is selected such that \eqref{8dm} is satisfied for $\tau\in[0,\tau_s]$ and $\hat{\tau}\in[\tau-\underline{\delta}, \tau+\bar {\delta}]$. There exists a value of the filter bandwidth $k$ such that the reference system is bounded-input bounded-output stable with respect to the desired trajectory $y_d(t)$ and the initial condition $x_0$.
		}
	\end{thm}
	
	\begin{proof}
		Taking the norm of \eqref{8riomap} yields
		\begin{align}
			\label{8orgineq}
			\|x_\mathrm{ref}(t)\|_{{\cal L}_\infty}<f(\tau,\hat\tau)\big(\theta_b\|x_\mathrm{ref}(t)\|_{{\cal L}_\infty}+\sigma_b\big)+\rho_d+\rho_\mathrm{ic}.
		\end{align}
		Since the stability condition in \eqref{8dm} is satisfied by appropriately selecting $k$ and for $\tau\in[0,\tau_s]$ and $\hat{\tau}\in[\tau-\underline{\delta}, \tau+\bar {\delta}]$, it follows that
		\begin{align}
			\label{8orgineq1}
			\|x_\mathrm{ref}(t)\|_{{\cal L}_\infty}<\frac{f(\tau,\hat\tau)\sigma_b+\rho_d+\rho_\mathrm{ic}}{1-f(\tau,\hat\tau)\theta_b}=:\rho_\mathrm{ref}
		\end{align}
		This and $y_\mathrm{ref}(t)=c^\top x_\mathrm{ref}(t)$ imply that the reference system is bounded-input bounded-output stable with respect to $y_d$ and $x_0$.
	\end{proof}
	
	\begin{rmk}
		When $\tau=\hat\tau=0$,
		\begin{align}
			\lim\limits_{s\to 0}&sc^\top \Big(\Psi(s;0,0)-H(s)\Big)k_dy_d(s)=\lim\limits_{s\to 0}c^\top H(s)\frac{-s^2}{s+k}k_dy_d(s).
		\end{align}
		Hence, for $y_d(s)$ with less than two zero poles, the above limit is equal to zero. In addition, we have
		\begin{equation}
			\lim_{s\rightarrow 0}s(s\mathbb{I}-A_m)^{-1}x_0=0.
		\end{equation}
		In such cases, it follows that, for large $t$
		\begin{align}
			\|y_\mathrm{ref}(t)-y_\mathrm{des}\|_\infty&\le c^\top f(0,0)\|\eta_\mathrm{ref}\|_{\mathcal{L}_\infty}\nonumber\\
			&<c^\top f(0,0)\left(\theta_b\frac{f(0,0)\sigma_b+\rho_d+\rho_\mathrm{ic}}{1-f(0,0)\theta_b}+\sigma_b\right).
		\end{align}
		
		Since $f(0,0)=\mathcal{O}(k^{-1})$, it follows that
		\begin{equation}
			\label{invproptok}
			\|y_\mathrm{ref}(t)-y_\mathrm{des}\|_\infty=\mathcal{O}(k^{-1})
		\end{equation}
		when $\tau=\hat\tau=0$ and $y_d(s)$ has less than two zero poles. No such conclusion follows otherwise.
	\end{rmk}
	
	\section{${\cal L}_1$ Adaptive Control Framework with Delay Compensation Modification}
	\label{8sec:transbnd}
	In this section, we first present the adaptive controller from \cite{Hovakimyan2010book} with a modification to mitigate the effect of the input delay for the system of interest. Let the control input $u(t)$ be the solution to the differential equation
	\begin{align}
		\label{8controllaw}
		\dot{u}(t)=-k\big(u(t)+\hat{\theta}(t)x(t)+\hat{\sigma}(t)+k_dy_d(t)\big),
	\end{align}
	where $u(t)=0\forall t\in[-\max\{\tau,\hat{\tau}\},0]$, and the functions $\hat{\theta}(t)$ and $\hat{\sigma}(t)$ are governed by the projection-based adaptive laws \cite{Hovakimyan2010book}
	\begin{align}
		\label{8adlaws}
		\dot{\hat{\theta}}(t)&=\Gamma\,\textbf{Proj}\big(\hat{\theta}(t),- \tilde{x}^\top(t)Pbx(t);\theta_b,\nu\big),\\\label{8adlaws2}
		\dot{\hat{\sigma}}(t)&=\Gamma\,\textbf{Proj}\big(\hat{\sigma}(t),- \tilde{x}^\top(t)Pb;\bar\sigma_b,\nu\big),
	\end{align}
	with $\hat{\theta}(0)=\hat{\theta}_0$ and $\hat{\sigma}(0)=\hat{\sigma}_0$, in terms of the estimation error $\tilde{x}:= \hat{x}-x$, where
	\begin{align}
		\label{8sp}
		\dot{\hat{x}}(t)&=A_mx(t)+A_{\mathrm{sp}}\tilde{x}(t)+b\big(u(t-\hat{\tau})+\hat{\theta}^\top(t)x(t)+\hat{\sigma}(t)\big)
	\end{align}
	and $\hat{x}(0)=x_{0}$. Here, the positive scalars $\Gamma>0$ and $k$ correspond to adaptive gain and the bandwidth of the first-order low-pass filter $k/(s+k)$. Moreover, $A_{\mathrm{sp}}$ is a Hurwitz matrix, which may be tuned to reduce any noise in the state predictor $\hat{x}$. The projection operators $\textbf{Proj}(\cdot,\cdot;\cdot,\cdot)$ are implemented in terms of the bounds $\theta_b$ and $\sigma_b$, the tolerance $\nu$, and the positive-definite matrix $P$, obtained as the solution to the Lyapunov equation \mbox{$A_{\mathrm{sp}}^\top P+PA_{\mathrm{sp}}=-\mathbb{I}$}. The implementation ensures that $\|\hat{\theta}(t)\|_\infty\leq\theta_b$ and $\|\hat{\sigma}(t)\|_\infty\leq\bar\sigma_b$ provided that $\hat{\theta}_0$ and $\hat{\sigma}_0$ satisfy these same bounds. The relation between $\sigma_b$ and $\bar\sigma_b$ will be defined later in the paper.
	
	The modification to the ${\cal L}_1$ adaptive control framework is the introduction of $\hat{\tau}$ in \eqref{8sp}. When  $\hat{\tau}=0$, the controller composed by \eqref{8controllaw}, \eqref{8adlaws}, \eqref{8adlaws2}, and \eqref{8sp} recovers the exact form of the framework proposed in \cite{Hovakimyan2010book}. 
	
	
	\section{Transient performance}
	\label{sec:trans}
	In this section, we analyze the transient performance of the closed-loop adaptive control system. The theorem below states that if $\tau\in[0,\tau_s]\mbox{ and }\hat{\tau}\in[\tau-\underline{\delta}, \tau+\bar {\delta}]$, and the bandwidth $k$, the adaptive gain $\Gamma$, and the bounds $\theta_b$ and $\bar\sigma_b$ are chosen appropriately, then the state and control input of the closed-loop control system governed by \eqref{inpdelay} and \eqref{8controllaw}-\eqref{8sp} follow those of the reference system closely. In particular,
	
	\begin{thm}
		\label{8thm2}
		\textit{Suppose $k$ is chosen as in Lemma~\ref{lem:stbcnd} and assume that $\tau\in[0,\tau_s]\mbox{ and }\hat{\tau}\in[\tau-\underline{\delta}, \tau+\bar {\delta}]$ such that \eqref{8dm} is satisfied. Then, there exists a $C>0$, such that, for $\psi\ll 1$, $\| \tilde{x}\|_{{\cal L}_{\infty}}\leq \psi$ and
			\begin{align}
				\label{8thmbs}
				\|x_\mathrm{ref}&-x\|_{{\cal L}_{\infty}},~
				\|u_\mathrm{ref}-u\|_{{\cal L}_{\infty}}
			\end{align}
			are bounded from above by $b_r\psi$ and $b_u\psi$ for some positive constants $b_r$ and $b_u$, provided that $\Gamma\psi^2\ge C$. }
	\end{thm}
	
	Before proving this theorem, we show that the state predictor tracks the system state with the estimation error inversely proportional to the square root of the adaptive gain.
	\begin{lem}
		\label{8lem:Lyap}
		Suppose that
		\begin{equation}
			\label{8rubnds}
			\|x_{t_1}\|_{{\cal L}_\infty}< \rho_\mathrm{ref}+1,~~\|u_{t_1}\|_{{\cal L}_\infty}<\rho_\mathrm{u}<\infty
		\end{equation}
		for some $t_1$ and $\rho_\mathrm{u}$. Then, there exists a $C>0$, which is independent of $t_1$, such that
		\begin{equation}
			\| \tilde{x}_{t_1}\|_{{\cal L}_{\infty}}\leq \sqrt{C/\Gamma}.
		\end{equation}
	\end{lem}
	
	\begin{proof}
		From \eqref{inpdelay} and \eqref{8sp}, we get
		\begin{align}
			\label{8esterror}
			\dot{\tilde{x}}(t)=A_{\mathrm{sp}}\tilde{x}(t)+b\big(\tilde{\theta}^\top(t)x(t)+\tilde{\sigma}(t)\big), ~\tilde{r}(0)=0,
		\end{align}
		where $\tilde{\theta}:=\hat{\theta}-\theta$, $\tilde{\sigma}:=\hat{\sigma}-\bar\sigma$, and
		\begin{equation}
			\bar\sigma(t)=\sigma(t)+u(t-\tau)-u(t-\hat{\tau}).
		\end{equation}
		It follows from the assumption on $u$ in \eqref{8rubnds} and the bound on $\sigma$ that 
		\begin{align}
			|\bar\sigma(t)|\le\sigma_b+2\rho_u=:\bar{\sigma}_b
		\end{align}
		with $\bar\sigma_b$ independent of $t_1$. Moreover,
		\begin{align}
			\label{8sigmabardot}
			\dot{\bar{\sigma}}(t)=\dot{\sigma}(t)+\dot{u}(t-\tau)-\dot u(t-\hat{\tau}).
		\end{align}
		By the assumptions on $x(t)$ and $u(t)$ in \eqref{8rubnds}, equations \eqref{inpdelay} and \eqref{8controllaw}-\eqref{8adlaws2} imply that $\dot{r}$ and $\dot{u}$ are bounded by constants independent of $t_1$. Hence, $\|\dot{\bar{\sigma}}(t)\|_\infty\leq d_{\bar{\sigma}}$, with $d_{\bar{\sigma}}$ independent of $t_1$.
		
		Consider the Lyapunov function candidate
		\begin{equation}
			V(t)=\tilde{x}^\top (t)P\tilde{x}(t)+\frac{1}{\Gamma}\big(\tilde{\theta}^\top(t)\tilde{\theta}(t)+\tilde{\sigma}^\top(t)\tilde{\sigma}(t)\big).
		\end{equation}
		By the properties of the projection operators,
		\begin{align}
			\label{8vdotineq}
			\dot{V}(t)&\leq-\tilde{x}^\top(t) \tilde{x}(t)+\frac{2}{\Gamma}\big|\tilde{\theta}^\top(t)\dot{\theta}(t)+\tilde{\sigma}^\top(t)\dot{\bar{\sigma}}(t)\big|\nonumber\\
			&\leq-\|\tilde{x}\|_2^2 +\frac{4}{\Gamma}\big(\theta_b d_\theta+\bar\sigma_b d_{\bar\sigma}\big).
		\end{align}
		
		We have
		\begin{equation}
			V(0)\le\frac{4}{\Gamma}\left(\theta_b^2+\bar\sigma_b^2\right)<\frac{\nu_m}{\Gamma},
		\end{equation}
		where $\nu_m:= 4(\theta_b^2+\bar\sigma_b^2)+4\lambda_{\mathrm{max}}(P)\big(\theta_b d_\theta+\bar\sigma_b d_{\bar\sigma}\big)$. We now show by contradiction that
		\begin{align}
			\label{8vineq}
			V(t)\leq \frac{\nu_m}{\Gamma},~~\forall t\in [0,t_1].
		\end{align}
		To this end, suppose that $V(\bar{t})>\nu_m/\Gamma$ and $\dot{V}(\bar{t})\ge 0$ for some $\bar{t}<t_1$. It follows that
		\begin{align}
			\frac{\nu_m}{\Gamma}<V(\bar t)&\leq\|\tilde{x}(\bar t)\|_2^2 \lambda_{\mathrm{max}}(P)+\frac{4}{\Gamma}\big(\theta_{b}^2+\bar{\sigma}_{b}^2\big).\nonumber
		\end{align}
		Hence,
		\begin{align}
			\label{8r2}
			\|\tilde{x}(\bar{t})\|_2^2>\frac{4}{\Gamma }\big(\theta_b d_\theta+\bar\sigma_b d_{\bar\sigma}\big).
		\end{align}
		
		By substituting \eqref{8r2} in \eqref{8vdotineq} we have $\dot{V}(\bar{t})<0$, which contradicts the assumption that $\dot{V}(\bar{t})\geq0$. Thus, $V(t)\leq \frac{\nu_m}{\Gamma}$ for all $t\in [0,t_1]$. Consequently,
		
		\begin{eqnarray}
			\label{8rtbtau}
			\|\tilde{x}_{t_1}\|_{{\cal L}_\infty}\leq \sqrt{\frac{\nu_m}{\lambda_{\mathrm{min}}(P)\Gamma}}\,.
		\end{eqnarray}
		The claim then follows.
	\end{proof}
	
	\begin{lem}
		\label{8lem:etatb}
		\textit{Let $\tilde\eta(t):= \tilde{\theta}^\top(t)x(t)+\tilde{\sigma}(t)$. When $\hat\tau\in[\tau-\underline{\delta}, \tau+\bar {\delta}]$, there exists a constant $b_0$ such that
			\begin{align}
				\| F(s;\tau,\hat\tau)\tilde\eta(s)\|_{{\cal L}_\infty}\leq b_0 \|\tilde x\|_{{\cal L}_\infty}.
			\end{align}
		}
	\end{lem}
	
	\begin{proof}
		From \eqref{8esterror} we have
		\begin{align}
			b^\top\dot{\tilde{x}}=b^\top A_{\mathrm{sp}}\tilde{x}+b^\top b\tilde\eta
		\end{align}
		and, consequently,
		\begin{align}
			\label{8etatildebound}
			\tilde\eta(s)=b^\ast(s\mathbb{I}-A_{\mathrm{sp}})\tilde{x}(s)
		\end{align}
		where the pseudoinverse $b^\ast=(b^\top b)^{-1}b^\top$, since $b^\top b$ is non-zero. Moreover, from \eqref{8urefLap}, it follows that
		\begin{align}
			s F(s;\tau,\hat\tau)=-k\Big(1+\big(\mbox{e}^{-\tau s}-\mbox{e}^{-\hat\tau s}+1\big) F(s;\tau,\hat\tau)\Big)
		\end{align}
		Since  $\| F(s;\tau,\hat\tau)\|_{{\cal L}_1}$ is bounded for $\hat\tau\in[\tau-\underline \tau,\tau+\bar \tau]$, it follows that the norm $\|s F(s;\tau,\hat\tau)\|_{{\cal L}_1}$ is bounded. Thus, \eqref{8etatildebound} yields
		\begin{align}
			\| F(s;\epsilon)\tilde\eta(s)\|_{{\cal L}_\infty}&\leq \|sF(s;\epsilon)b^\ast \tilde{x}(s)\|_{{\cal L}_\infty}+\|F(s;\epsilon)b^\ast A_\mathrm{sp} \tilde{x}(s)\|_{{\cal L}_\infty}
		\end{align}
		and the claim follows.
	\end{proof}
	
	We proceed to prove Theorem~\ref{8thm2}.
	
	\begin{proof}
		Since
		\begin{equation}
			\|x_\mathrm{ref}(0)-x(0)\|_\infty=0<1,~~\|u_\mathrm{ref}(0)-u(0)\|_\infty=0,
		\end{equation}
		it follows by continuity that there exists a $t_1>0$, such that $\|(x_\mathrm{ref}-x)_{t_1}\|_{{\cal L}_\infty}<1$ and $\|(u_\mathrm{ref}-u)_{t_1}\|_{{\cal L}_\infty}<\infty$. Theorem~\ref{8thm1} then leads to \eqref{8rubnds}. It follows that
		\begin{equation}
			\| \tilde{x}_{t_1}\|_{{\cal L}_{\infty}}\leq \sqrt{C/\Gamma}.
		\end{equation}
		for some $C>0$, which is independent of $t_1$.
		
		Next, we write the control law in \eqref{8controllaw} in a similar form as in the control law \eqref{uref} of the reference system:
		\begin{align}
			\dot{u}(t)&=-k\Big(u(t-\tau)-u(t-\hat{\tau})+u(t)\nonumber\\
			&+\big(\hat{\theta}(t)-\theta(t)\big)^\top x(t)+\hat{\sigma}(t)\underbrace{-\sigma(t)-u(t-\tau)+u(t-\hat{\tau})}_{-\bar{\sigma}}+\underbrace{\theta^\top(t)x(t)+\sigma(t)}_{\eta(t)}+k_dy_d(t)\Big)
		\end{align}
		and, consequently,
		\begin{align}
			\label{8u2}
			u(s)&= F(s;\tau,\hat{\tau})\big(\tilde \eta(s)+\eta(s)+k_dy_d(s)\big).
		\end{align}
		It follows from \eqref{8usdef} and \eqref{8u2}  that
		\begin{equation}
			\label{8deltau}
			u_\mathrm{ref}(s)-u(s)= F(s;\tau,\hat{\tau})\big(\eta_\mathrm{ref}(s)-\eta(s)-\tilde \eta(s)\big).
		\end{equation}
		It further follows from the definitions of $\eta_\mathrm{ref}$ and $\eta$ that
		\begin{equation}
			\label{8etabound}
			\|(\eta_\mathrm{ref}-\eta)_{t_1}\|_{{\cal L}_{\infty}}\leq \theta_b\|(x_\mathrm{ref}-x)_{t_1}\|_{{\cal L}_{\infty}}
		\end{equation}
		and, using Lemma~\ref{8lem:etatb},
		\begin{align}
			\label{8deltaunorm}
			\|(u_\mathrm{ref}&-u)_{t_1}\|_{{\cal L}_\infty}\leq g(\tau,\hat{\tau}) \theta_b\|(x_\mathrm{ref}-x)_{t_1}\|_{{\cal L}_{\infty}}+b_0\| \tilde{x}_{t_1}\|_{{\cal L}_{\infty}}.
		\end{align}
		
		From \eqref{inpdelay}, we now obtain
		\begin{align}
			\label{8deltar}
			x_\mathrm{ref}(s)-x(s)&= H(s)\Big(\mbox{e}^{-\tau s}\big(u_\mathrm{ref}(s)-u(s)\big)+\eta_\mathrm{ref}(s)-\eta(s)\Big).
		\end{align}
		Together with \eqref{8deltau}--\eqref{8etabound} and Lemma~\ref{8lem:etatb}, this results in the bound
		\begin{align}
			\label{8r_refr_norm}
			\|(x_\mathrm{ref}&-x)_{t_1}\|_{{\cal L}_{\infty}}\leq f(\tau,\hat{\tau}) \theta_b\|(x_\mathrm{ref}-x)_{t_1}\|_{{\cal L}_{\infty}}+b_2\| \tilde{x}_{t_1}\|_{{\cal L}_{\infty}},
		\end{align}
		where $b_2:= b_0\| H(s)\mbox{e}^{-\tau s}\|_{{\cal L}_1}$, which is finite for $\hat{\tau}\in[\tau-\underline{\delta}, \tau+\bar {\delta}]$.

		When $\tau\in[0,\tau_s]\mbox{ and }\hat{\tau}\in[\tau-\underline{\delta}, \tau+\bar {\delta}]$, the stability condition in \eqref{8dm} holds and implies that $1-f(\tau,\hat{\tau}) \theta_b>0$. Thus, from \eqref{8deltaunorm} and \eqref{8r_refr_norm}, we conclude that
		\begin{align}
			\label{8r_refr_norm1}
			\|(x_\mathrm{ref}-x)_{t_1}\|_{{\cal L}_{\infty}}\leq\frac{b_2}{1-f(\tau,\hat{\tau}) \theta_b}\| \tilde{x}_{t_1}\|_{{\cal L}_{\infty}}
		\end{align}
		and
		\begin{equation}
			\|(u_\mathrm{ref}-u)_{t_1}\|_{{\cal L}_\infty}\leq\left(\frac{g(\tau,\hat{\tau})b_2\theta_b}{1-f(\tau,\hat{\tau}) \theta_b}+b_0\right)\| \tilde{x}_{t_1}\|_{{\cal L}_{\infty}}.
		\end{equation}
		The claim then follows by choosing $\psi$ and $\Gamma$ such that the right-hand side of \eqref{8r_refr_norm1} is strictly less than 1 and $\sqrt{C/\Gamma}<\psi$.
	\end{proof}
	
	In the next section, we will demonstrate the destabilizing effect of input delay $\tau$ as well as the efficacy of $\hat{\tau}$. In addition, we use Pad\'{e} approximants to visualize the stability condition in \eqref{8dm} and to estimate ranges $\tau_s$ and $\underline{\delta}$, and $\bar{\delta}$ for the time delay margin by monitoring the stability condition in \eqref{8dm} when the delay is being gradually increased.
	\section{Numerical Example}
	\label{8sec:num}
	\subsection{Simulation}
	Consider the controller with the delay compensation in \eqref{8controllaw}-\eqref{8sp} applied to system \eqref{inpdelay} with (cf.~\cite{Hovakimyan2010book})
	\begin{align}
		\label{param}
		A_m&=\left( \begin{array}{cc}
			0 & 1 \\
			-1 & -1.4 \end{array} \right), ~~b=\left( \begin{array}{c}
			0 \\
			1\end{array} \right),~c=\left( \begin{array}{c}
			1 \\
			0\end{array} \right)\nonumber\\
		\theta(t)&=[0.5+\cos(\pi t), ~~1+0.3\sin(\pi t)+0.2 \cos(2t)]^\top,\nonumber\\
		\sigma(t)&=\sin(0.5\pi t).
	\end{align}
	Since it is trivial to see that $\|\theta(t)\|_2<2$, we set $\theta_b=2$. The control parameters are set to: 
	$\Gamma=10^7$, $k=25$, $\bar\sigma_b=100$, $\hat\theta_0=
	\hat\sigma_0=0$, $\nu=0.1$, and $A_\mathrm{sp}=100\mathbb{I}$. The states of the actual adaptive system and the desired system are initialized at $x_0=[0,~1]^\top$ and $x_{\mathrm{des},0}=[1,~0]^\top$, respectively. The system output is tasked to track the output of a desired system $y_\mathrm{des}$ with $y_d=\cos (2t/\pi)$.
	
	We first set the input delay $\tau=0.06$ and turn off the delay compensation, i.e., $\hat{\tau}=0$. Figure~\ref{fig:delayp5} shows the result of the simulation. We can see that the output $y(t)$ is able to track closely the output of the desired system $y_\mathrm{des}$. Even without the delay compensation, the ${\cal L}_1$ controller is able to tolerate some small input delay in the limit of large adaptive gain. This is consistent with the fact that the ${\cal L}_1$ controller is robust with fast adaptation as widely reported in the literature. The observed deviation between the two trajectories is traced back to the need to keep a finite filter bandwidth $k$ in order to maintain system robustness. In the absence of the input delay, the deviation gets smaller when $k$ is increased. This agrees with the result in \eqref{invproptok}. However, increasing $k$, while improving tracking performance, deteriorates the controller's delay robustness. Therefore, there is a trade-off between tracking performance and system robustness in the ${\cal L}_1$ control framework.
	\begin{figure}[ht]
		\centering
		\includegraphics[width=.9\textwidth]{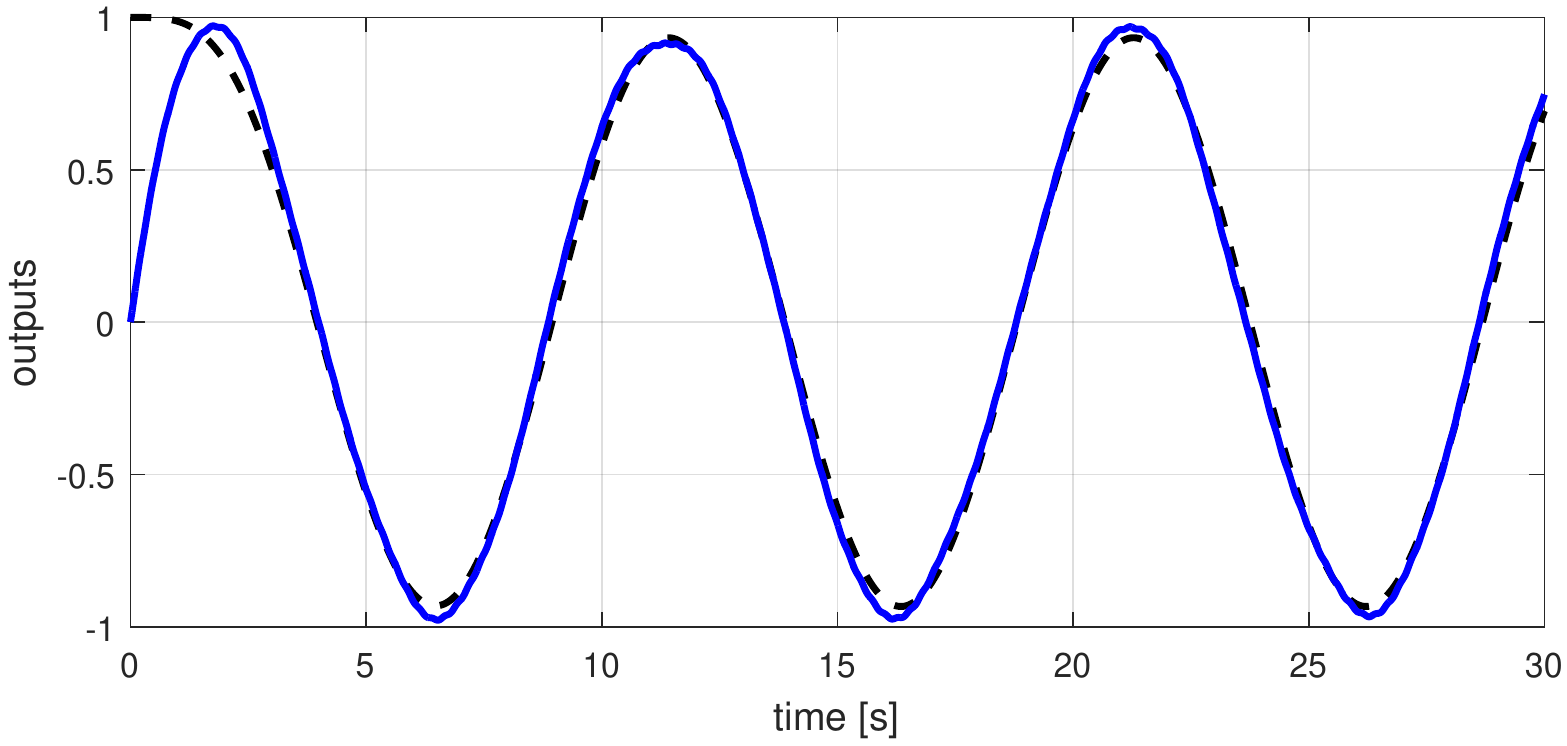}
		\caption{Satisfying response with input delay $\tau=0.06$.}
		\label{fig:delayp5}
	\end{figure}
	
	
	When the input delay is large enough, the tracking performance of the unmodified ${\cal L}_1$ control system is degraded. For instance, with delay $\tau=0.07$, the system output deviates from the desired trajectory and exhibits high-frequency oscillations as seen in Figure~\ref{fig:delayp52}. Further increases in input delay eventually result in unbounded system output. 
	
	\begin{figure}[h]
		\centering
		\centering
		\includegraphics[width=.9\textwidth]{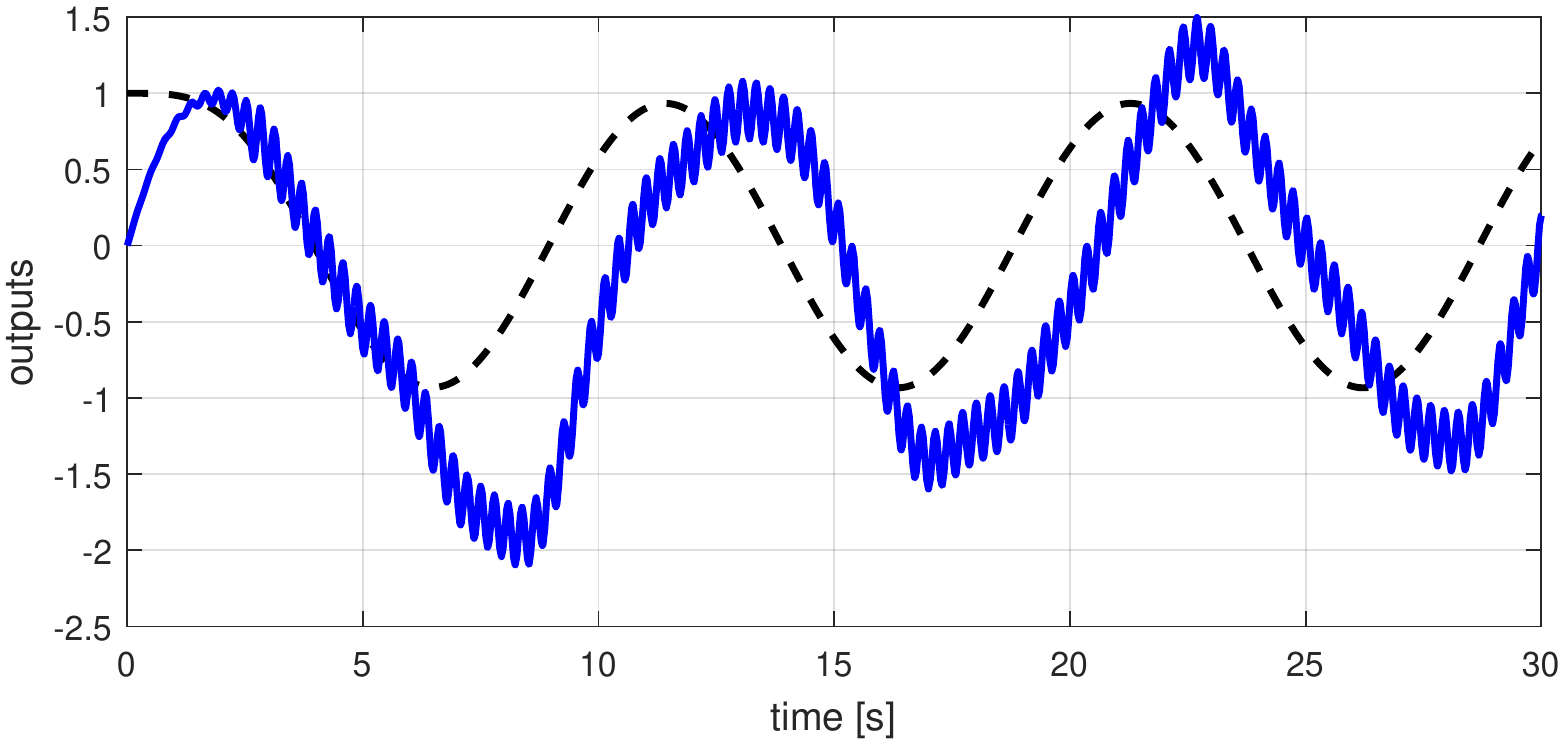}
		\caption{Response with input delay $\tau=0.07$: high frequency oscillations with large tracking errors.}
		\label{fig:delayp52}
	\end{figure}
	
	We next illustrate the efficacy of the delay compensation modification represented by $\hat\tau$ in \eqref{8sp}. The input delay $\tau$ is kept at $0.07$, while $\hat\tau$ is set to $0.02$. The resultant output in Figure~\ref{fig:wtauhat} shows that the system recovers desired performance with small tracking error in the presence of the compensation delay $\hat\tau$. In fact, the control system remains stable for values of $\tau$ as high as $0.43$ provided that $\hat\tau$ is tuned to a value near $0.43$.
	\begin{figure}[!h]
		\centering
		\includegraphics[width=.9\textwidth]{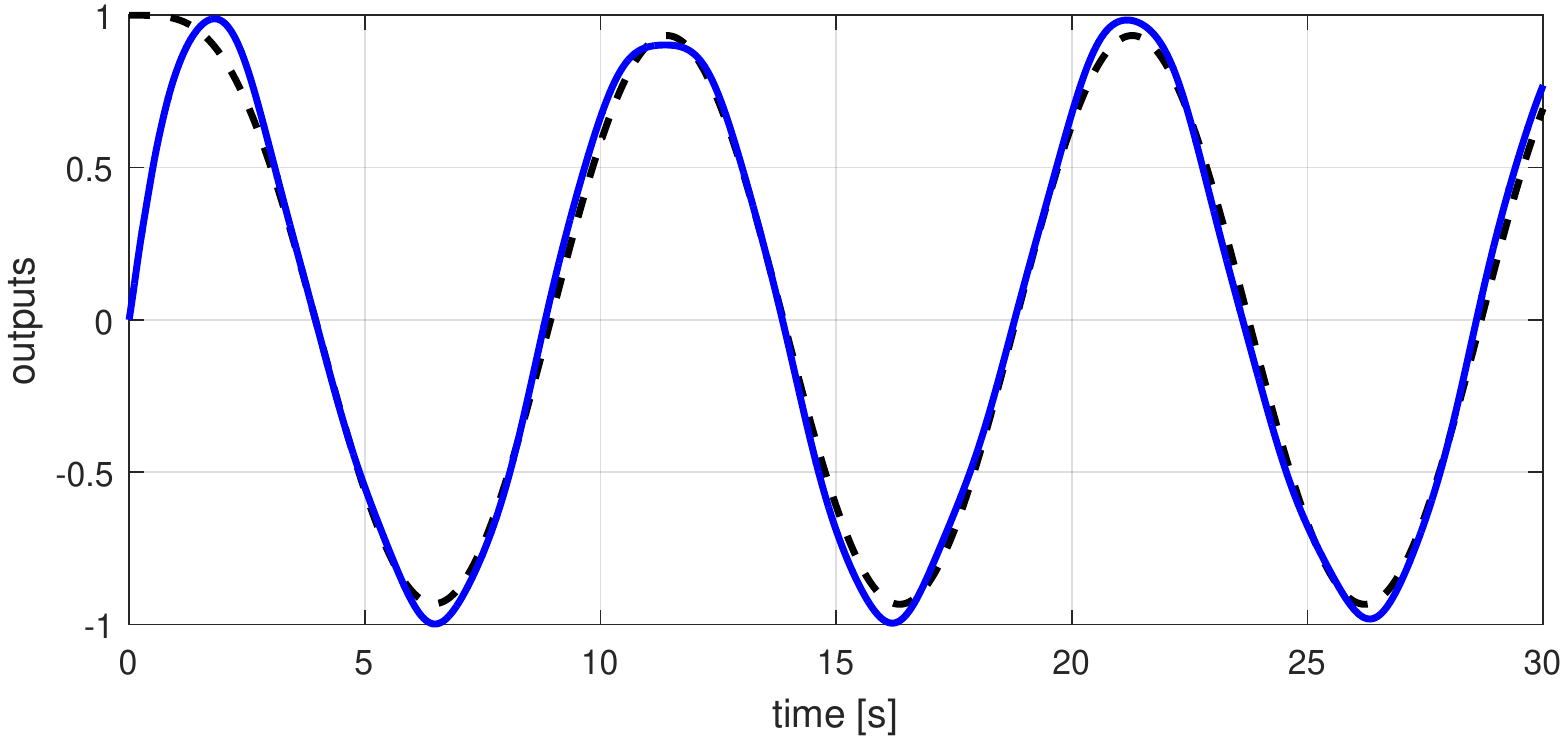}
		\caption{Response with input delay $\tau=0.07$ and $\hat\tau=0.02$: desired tracking is recovered}
		\label{fig:wtauhat}
	\end{figure}
	\subsection{Stability chart via a continuation approach}
	In this section, we discuss a method for numerically analyzing the sufficient stability condition in \eqref{8L1cond}. This method will result in a stability chart based on techniques of parameter continuation. The result will also demonstrate the effect of the compensation delay as well as the dependence of the delay compensation performance on the low-pass filter bandwidth.

	It is practically impossible to determine the explicit form of the impulse response for the transfer function in the definition of $f(\tau,\hat\tau)$ in \eqref{8dm} due to the delay terms $\mbox{e}^{-\tau s}$ and $\mbox{e}^{-\hat\tau s}$, which result in infinite number of poles of the transfer function. To overcome this challenge, we employ a $(5,5)$ Pad\'{e} approximant to arrive at the following approximation
	\begin{align}
		\label{padedelay}
		e^{-\tau s}\approx\frac{\sum\limits_{i=0}^{5} (-1)^ic_i\tau^is^i}{\sum\limits_{j=0}^{5} c_j\tau^js^j},
	\end{align}
	where
	\begin{align}
		\label{padecoef}
		c_i=\frac{(10-i)!5!}{(10)!i!(5-i)!}.
	\end{align}
	
	The convergence and the error bounds of Pad\'{e} approximants are discussed in \cite{Lam1990}. A Pad\'{e} approximant allows for the expansion of the delay terms as rational transfer functions. This enables the estimation of the norm in the definition of $f(\tau,\hat{\tau})$ in \eqref{8fd}, given values of $\tau$ and $\hat{\tau}$, via the algorithm proposed in \cite{Neil1995} with a tolerant of $10^{-5}$. The control parameters remain unchanged from the previous simulations. While increasing $\tau$ from zero, for each value of the input delay $\tau$, we vary $\hat{\tau}$ from $\tau$ until the stability condition \eqref{8dm} is violated.

    We compute the stability chart for all $\tau$ and $\hat\tau$ that satisfy the stability condition in \eqref{8L1cond} for the example system of interest in this section. This is performed using the continuation toolbox called COCO \cite{Dankowicz2013}. The continuation tolerant is set to $10^{-5}$. For a given value of the filter bandwidth $k$, we compute the implicitly defined solution manifold of the equation
    \begin{align}
    \label{zerofunc}
    f(\tau,\hat{\tau})-0.5=0
    \end{align}
    
    Figure~\ref{fig:k25} shows the result for a bandwidth $k=25$. The solid line represents the pairs $(\tau,\hat{\tau})$ values that satisfy \eqref{zerofunc}. The use of the continuation techniques allows for the automatic computation of the complicated folds on the curve, which may be highly challenging for traditional root-finding methods, for example, the bisection method. On one side of the curve, the stability condition holds and the stability of the modified controller is guaranteed. In contrast, on the other side of the manifold, the condition is violated and stability is no longer guaranteed. The dash-dot line indicates the identity line. 
    
    The continuation result confirms the stability behavior suggested by the theoretical analysis in Section~\ref{sec:trans}. In particular, along the identity line, $\hat\tau=\tau$, there is a $\tau_s$ such that when $\tau<\tau_s$, the system stability is guaranteed. For the solution manifold in Figure~\ref{fig:k25} with $k=25$, we obtain $\tau_s=0.212$. In fact, the manifold suggests that stability is guaranteed for even higher values of $\tau$ off the identity line. Hence, the largest possible $\tau$ that guarantees stability if predicted by the theoretical analysis may be conservative.
    
    In addition, the numerical results also validate this paper's claim that for a given $\tau\leq\tau_s$, there are $\bar{\delta}$ and $\underline{\delta}$ such that the system stability is guaranteed if $\hat{\tau}\in [\tau+\bar{\delta}, \tau-\underline{\delta}]$. For example, in Figure~\ref{fig:k25} with $k=25$, at $\tau=0.15$, we have $\bar{\delta}=0.206$ and $\underline{\delta}=0.071$. In addition, when $\tau<0.042$, the value band $\bar{\delta}+\underline{\delta}$ for $\hat{\tau}$  is significantly larger as compared to that when $\tau\ge 0.042$. When $\tau$ is around $0.042$, the manifold exhibits a dramatic change and the value of $\bar{\delta}$ varies greatly. Another observation is that when $\tau<0.056$, the value of $\tau-\underline{\delta}$, represented by the brown solid line in Figure~\ref{fig:k25}, equals zero since $\hat{\tau}$ is non-negative. This brown solid line also indicates the lower bound for time-delay margin of the adaptive control system in the absence of the delay-compensation modification.
    \begin{figure*}[h!]
    	\centering
    	\includegraphics[width=1\textwidth]{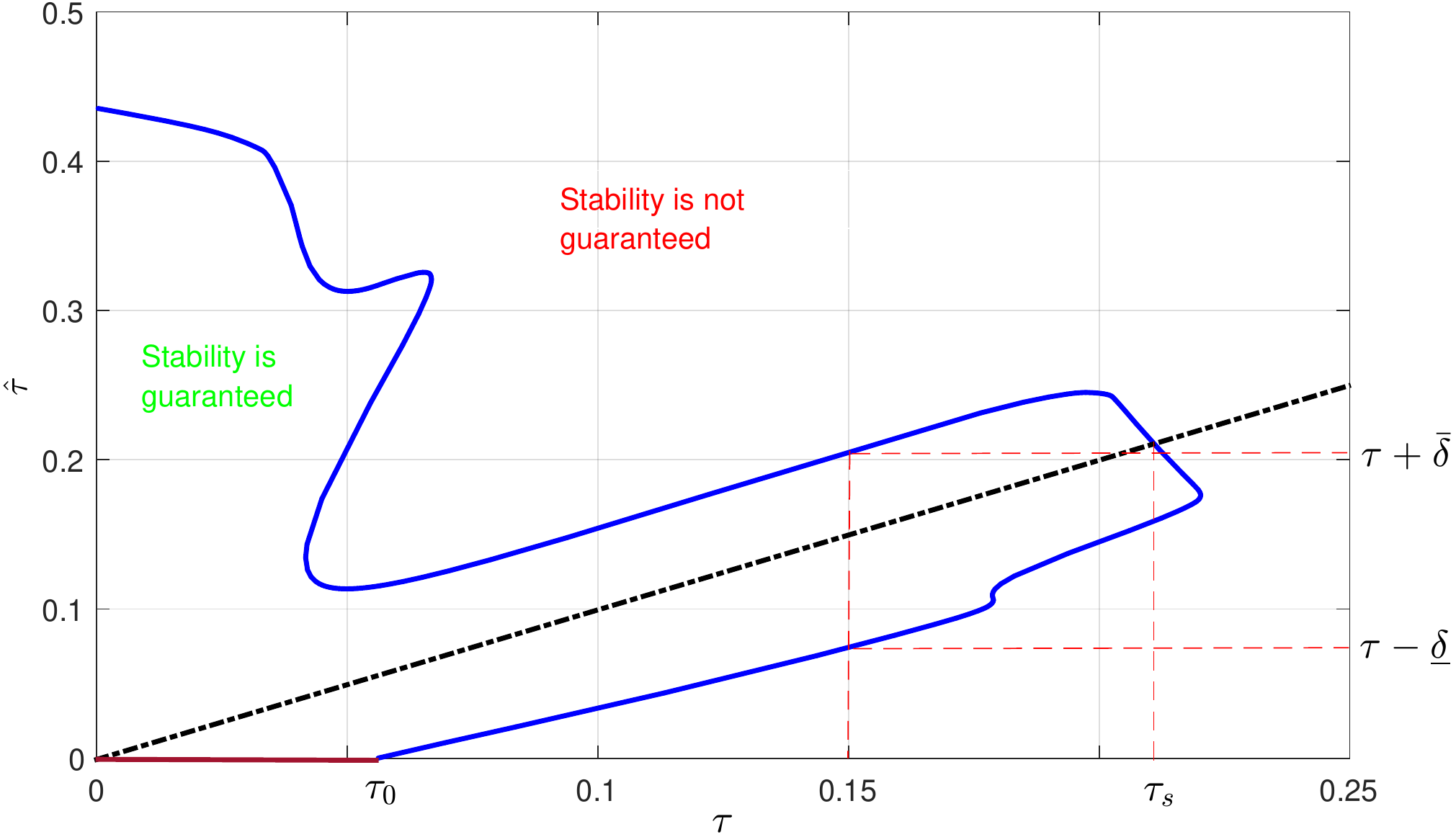}
    	\caption{Stability chart of the adaptive control system with $k=25$}
    	\label{fig:k25}
    \end{figure*}    
    
    To study the dependence of the stability chart on the filter bandwidth $k$, the same continuation process is performed for $k=50$ and $100$. The results, given by Figures~\ref{fig:k50} and \ref{fig:k100}, respectively, illustrate several points. First, in the presence of the delay compensation $\hat{\tau}$, the time-delay margin of the unmodified controller, indicated by the brown solid lines in Figures~\ref{fig:k25}, \ref{fig:k50}, and \ref{fig:k100}, is decreased as the bandwidth increases. This is consistent with the result in \cite{Nguyen2015tdm}. When the input delay is near zero, the maximum allowable $\hat{\tau}$ is large and increasing with increasing bandwidth $k$. This implies that over treatment when the delay is small is not a problem since stability is guaranteed even for a large deviation between the compensation delay and the actuator delay. 
    
    Furthermore, for input delay $\tau$ far away from zero, the value band of $\hat{\tau}$ that guarantees stability gets narrower for higher bandwidth. For example, when $\tau=0.15$, the allowable value band $\bar{\delta}+\underline{\delta}$ for $\hat{\tau}$ is $0.135$, $0.068$, and $0.053$ with $k=25$, $50$, and $100$, respectively. Another point to remark is that increasing filter bandwidth surprisingly improves the time-delay margin of the adaptive system along the identity line in the $(\tau,\hat{\tau})$ plane. Specifically, the time-delay margin along the identity line $\tau_s=0.212,~0.225$, and $0.233$ with $k=25$, $50$, and $100$, respectively. This result implies an interesting property of the modified controller. In the unmodified ${\cal L}_1$ control framework, increasing the filter bandwidth deteriorates system robustness (see \cite{Nguyen2015tdm} for the detailed analysis). In contrast, with the delay-compensation modification, when it is possible to set $\hat{\tau}=\tau$, increasing bandwidth indeed improves system robustness, indicated by the ability to tolerate input delay.
	
		\begin{figure*}
			\centering
			\includegraphics[width=1\textwidth]{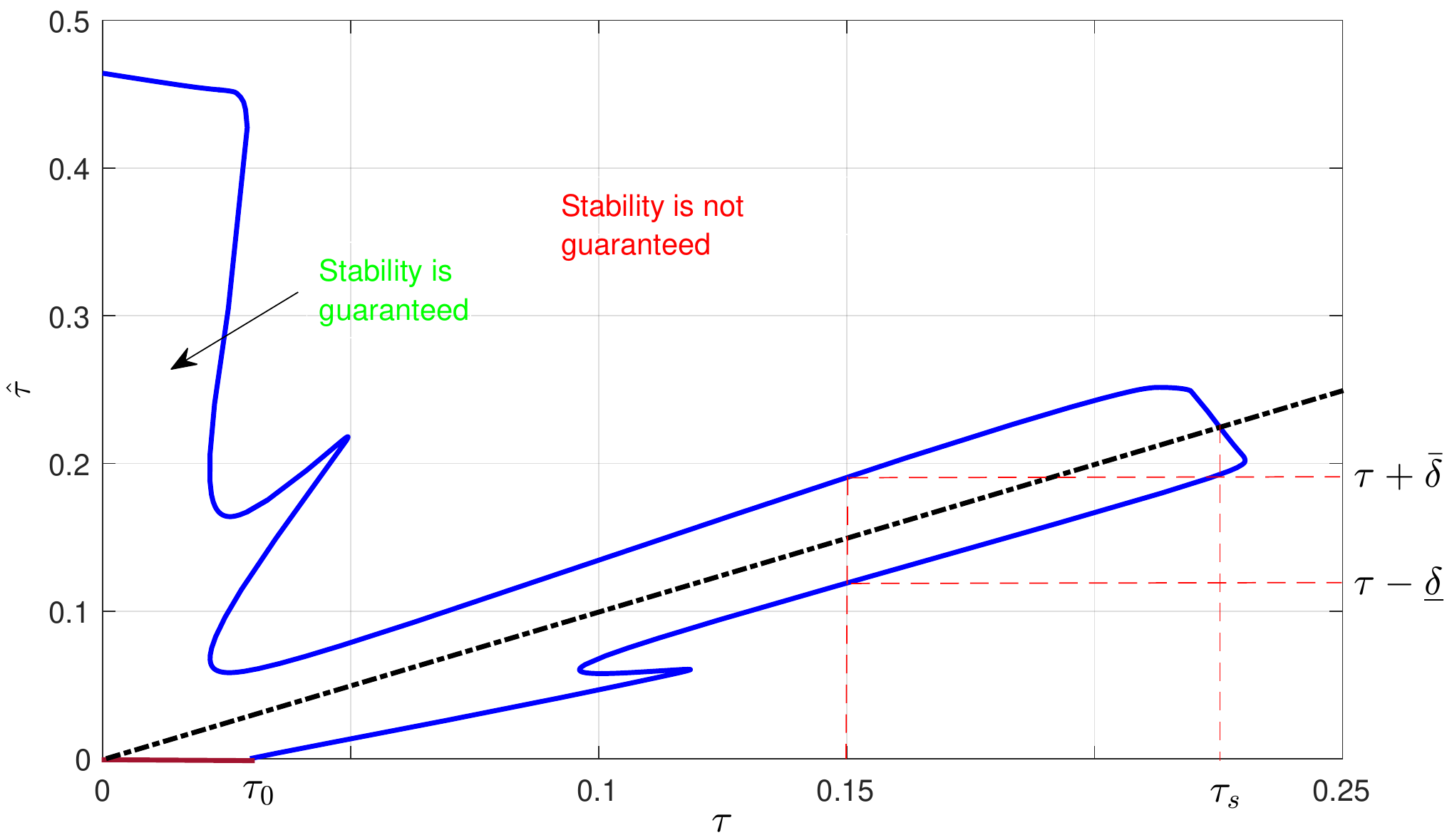}
			\caption{Stability chart of the adaptive control system with $k=50$}
			\label{fig:k50}
		\end{figure*}	
		\begin{figure*}
			\centering
			\includegraphics[width=1\textwidth]{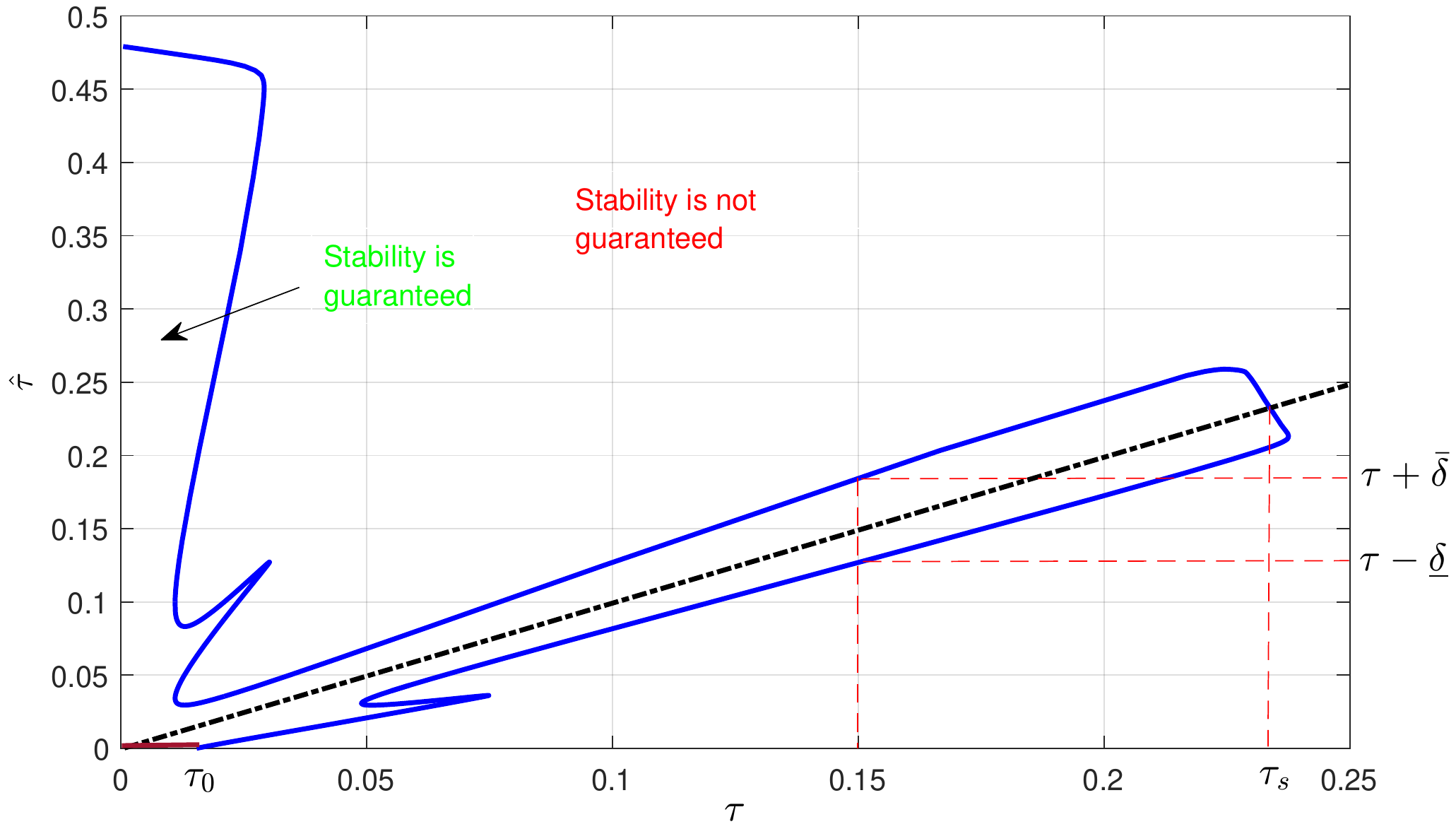}
			\caption{Stability chart of the adaptive control system with $k=100$}
			\label{fig:k100}
		\end{figure*}

	\section{Conclusion}
	\label{8conc}
	We have suggested and analyzed a modification to a controller with fast adaptation to compensate for an actuator delay in the plant. The idea is simple: we inject a delay in the control input of the predictor so that it emulates the structure of the actual plant equation. We first prove that a nonadaptive reference system is bounded-input-bounded-output stable provided that a delay-dependent stability condition satisfied. In that case, when the adaptive gain is chosen sufficiently large, the adaptive system output is shown to follow closely that of the nonadaptive reference system. The stability condition may be satisfied for a range of values of the compensation delay about the input delay, for sufficiently large filter bandwidth and sufficiently small values of the input delay. The set of values of the delays that satisfy the stability condition may be estimated numerically using Pad\'{e} approximants.

	The efficacy of the delay-compensation term has not been rigorously proved though it is demonstrated by the illustrative example and numerical analysis. The method of proof does not produce an explicit estimate for the upper bound on the range of input delay $\tau$ for which stable behavior may be achieved. In addition, a closed-form expression for the allowable range of $\hat{\tau}$ around $\tau$ for the system stability has not been formulated. Furthermore, the predictable dependence of the system's time-delay margin on $\hat{\tau}$ or $\hat{\tau}-\tau$ has not been completely investigated. The main challenge is how to obtain an analyzable form of the ${\cal L}_1$ norm in terms of the delays in the stability condition \eqref{8dm}. Finding a way to overcome this challenge will be the focus of future work.
	
\section*{Acknowledgement}
This work is supported by Agriculture and Food Research Initiative Competitive Grant no. 2014-67021-22109 from the USDA National Institute of Food and Agriculture.

\end{document}